\documentclass[11pt,reqno]{amsart}
\usepackage{amssymb,amsmath,amsfonts}
\usepackage{mathrsfs}
\usepackage[margin=1in]{geometry}
\usepackage{graphicx}
\usepackage{relsize}
\usepackage{cite}
\usepackage{color,xcolor}
\definecolor{cobalt}{RGB}{61,99,181}
\usepackage[colorlinks, citecolor=cobalt, linkcolor=cobalt]{hyperref}

\allowdisplaybreaks

\newtheorem{thm}{Theorem}[section]

\newtheorem{lem}[thm]{Lemma}
\newtheorem{rem}[thm]{Remark}

\newtheorem{prop}[thm]{Proposition}

\newtheorem{prob}[thm]{Problem}

\numberwithin{equation}{section}
\def\D{\mathbb{D}}
\def\T{\mathbb{T}}

\makeatletter
\newcommand{\rmnum}[1]{\romannumeral #1}
\newcommand{\Rmnum}[1]{\expandafter\@slowromancap\romannumeral #1@}
\@namedef{subjclassname@2020}{\textup{2020} Mathematics Subject Classification}
\makeatother

\begin{document}

\title[Cauchy singular integral operators]{Essentially semi-commuting singular integral operators with Cauchy kernel on $L^{2}$}

\author[Chongchao Wang]{Chongchao Wang}
\address{College of Mathematics and Physics, Wenzhou University, Wenzhou 325035, People's Republic of China}
\email{chchaowang@wzu.edu.cn}

%\author[Xianfeng Zhao]{Xianfeng Zhao\textsuperscript{2,3}}
%\address{\textsuperscript{2} College of Mathematics and Statistics, Chongqing University, Chongqing 401331, People's Republic of China  and  Key Laboratory of Nonlinear Analysis and its Applications (Chongqing University), Ministry of Education}
%\address{\textsuperscript{3} Key Laboratory of Nonlinear Analysis and its Applications (Chongqing University), Ministry of Education, Chongqing 401331, People's Republic of China}
%\email{xianfengzhao@cqu.edu.cn}

\date{\today}
\thanks{The author was supported by Zhejiang Provincial Natural Science Foundation (grant number: ZCLQ24A0101). }
\begin{abstract}
In this paper, we completely characterize when the semi-commutator of two Cauchy singular integral operators on $L^2$ is compact. Our main idea  is to study Cauchy singular integral operators via  Hankel operators,  Toeplitz operators  and function algebras. Moreover, we obtain a necessary and sufficient condition for Cauchy singular integral operators to be essentially normal. 
\end{abstract}

\maketitle

\section{Introduction}
Let $\D$ be the open unit disk and $\partial \mathbb D$ be its boundary. Let $L^2$ denote the Lebesgue space of square integrable functions on the unit circle $\partial \mathbb D$. The Hardy space $H^2$ is the closed subspace of $L^2$, which is spanned by the space of analytic polynomials. Thus there is an orthogonal projection $P$ from $L^2$ onto $H^2$. For $\varphi$ in $L^\infty$, the space of essentially bounded  measurable functions on  $\partial \mathbb D$, the Toeplitz operator $T_\varphi$ and the Hankel operator $H_\varphi$ with symbol $\varphi$ on $H^2$ are defined by
$$T_\varphi f=P(\varphi f)$$
and
$$H_\varphi f=(I-P)(\varphi f)$$
for $f\in H^2$, respectively. Moreover, the dual Toeplitz operator $S_\varphi$ on $(H^2)^{\bot}$  is defined by
$$S_\varphi h= (I-P)(\varphi h),  \ \ \ \  h\in (H^2)^{\bot}.$$
For more information on the topics of Toeplitz and Hankel operators we refer to \cite{Dou, Zhu}.

For $f, g\in L^{\infty}$, the singular integral operator $S_{f,g}$  with Cauchy kernel on $L^{2}$ is defined by
$$S_{f,g}h=fPh+g\left(I-P\right)h,\~~~\~~~\left(h\in L^2\right).$$ 
Then 
$$\left(S_{f,g}h\right)(z)=\frac{f(z)+g(z)}{2}h(z)+\frac{f(z)-g(z)}{2}\frac{1}{\pi i}\int_{\T}\frac{h(\tau)}{\tau-z}d\tau,$$
where the integral is understood in the sense of Cauchy’s principal value, see \cite[Vol. I, p.12]{GoK} for the details.

The Cauchy singular integral operator plays an important role in function theory and operator theory. There are many significant properties of such operators have been studied over the past four decades in \cite{Na1,Na2,Na3,NT,NY1,NY2,NY3,NY4,Ya1,Ya2}, such as the norm, invertibility, boundedness, etc. In 2014, Nakazi and Yamamoto \cite{NY} characterized the normal and self-adjoint Cauchy singular integral operators,  which have analogous properties to the Toeplitz operators. Besides, some other algebraic properties were studied   
by Gu \cite{Gu}, Samanta and Sarka \cite{SS}. In particular, they concretely characterized commutativity of two Cauchy singular integral operators. 

The problem of when the semi-commutator or commutator of two operators  is compact on function spaces  has been investigated by many people.   The beautiful Axler-Chang-Sarason-Volberg theorem (\cite{ACS}, \cite{V}) states that   the semi-commutator $ T_{f}T_{g}-T_{fg}$  of two Hardy Toeplitz operators $T_f$ and $T_g$ is compact if and only if  either $\overline{f}$ or $g$ is in $H^\infty$ on each \emph{support set} (which will be introduced in the next section). An elementary characterization for the compactness of the semi-commutator of two Hardy Toeplitz operators in terms of Hankel operators was obtained by Zheng \cite{Zheng}. The compactness for the semi-commutator of two Toeplitz operators on other analytic function spaces has been studied in \cite{GuZ}, \cite{GSZ} and \cite{Zheng1}. In 1999, Gorkin and Zheng \cite{GpZ}   completely characterized the compact commutator $T_fT_g-T_gT_f$ of two Toeplitz operators on the Hardy space in terms of Douglas algebras or support sets.   More precisely, the characterization in \cite{GpZ} can be stated as follows: two Toeplitz operators are essentially commuting if and only if either the restrictions of their symbols on each support set $S$ are in $H^\infty|_{S}$, or the restrictions of the conjugations of their symbols on each $S$ belong to $H^\infty|_{S}$, or a nontrivial linear combination of the restrictions of their symbols on  each support set $S$ is constant. Recently, Zou and Zhao \cite{ZZ} completely characterized the essential commutativity of two Cauchy singular integral operators on $L^{2}$.

In the present paper, we focus on the following  problem:
\begin{prob}
When is   the semi-commutator $\left[S_{f_1,g_1},S_{f_2,g_2}\right)=S_{f_1,g_1}S_{f_2,g_2}-S_{f_1f_2,g_1g_2}$ of two Cauchy singular integral operators $S_{f_1,g_1}$ and $S_{f_2,g_2}$ with $f_1$, $f_2$, $g_1$ and $g_2$ in $L^\infty$
compact?
\end{prob}

In order to study Cauchy singular integral operators, we use the useful matrix representation for Cauchy singular integral operator to establish a connection between the Toeplitz operator,  Hankel operator and Cauchy singular integral operator. Then the above essentially semi-commuting problem can be reduced to the study of the compactness of products of Toeplitz, Hankel and dual Toeplitz operators.
 The  difficult part in this paper is characterizing the compactness of the sum of the four products of Toeplitz, Hankel and dual Toeplitz operators. Our main idea here is to study singular integral operators via the characterization for the essentially commuting Hankel and Toeplitz operators \cite{GkZ1} and function algebras. The main result in this paper is the following theorem:
\begin{thm}\label{MR}
Let $f_1, f_2, g_1,g_2\in L^{\infty}$. Then the semi-commutator $S_{f_1,g_1}S_{f_2,g_2}-S_{f_1f_2, g_1g_2}$ is compact if and only if for each support set $S$, one of the following holds:\\
$(1)$ $\left(f_1-g_1\right)|_{S}=0$;\\
$(2)$ $f_2|_{S}$ and $\overline{g_2}|_{S}$ are in $H^{\infty}|_{S}$.
\end{thm}

Theorem \ref{MR} is analogous to the characterization for the semi-commutativity of two Cauchy singular integral operators (see \cite{SS}).

We organize this paper as follows. In Section 2, we introduce some notations and necessary lemmas. As the proof of Theorem \ref{MR} is long, it is divided into the necessary part and the sufficient part in Section 3. In Section 4, we  characterize the essentially normality of Cauchy singular integral operator.

\section{Notations and preliminaries}
In this section, we introduce some notations and include some important lemmas. Let us  begin with the following matrix representation for the Cauchy singular integral operator on $L^{2}$, which is introduced in \cite{NY}.

\begin{lem}\label{lem1}\cite[Lemma 1.1]{NY}
Let $f$ and $g$ be functions in $L^{\infty}$. Then $L^2=H^2\oplus \left(H^2\right)^{\bot}$ is a decomposition
such that, as matrices relative to this decomposition,
\[S_{f,g}=
\left(\begin{array}{cccc}
    T_{f} & H_{\overline{g}}^* \\
    H_{f} & S_{g}
\end{array}\right),\ \ \ \
S^{*}_{f,g}=\left(\begin{array}{cccc}
    T_{\overline{f}} & H_{f}^* \\
    H_{\overline{g}} & S_{\overline{g}}
\end{array}\right).
\]
\end{lem}

In view of the matrix representation in the above lemma, the essentially semi-commuting  problem for two Cauchy singular integral operators can be easily transformed into the compactness of the following four classical operators.

\begin{lem}\label{lem2}
Let $f_1$, $f_2$, $g_1$ and $g_2$ be functions in $L^{\infty}$. Then the semi-commutator 
$$S_{f_1,g_1}S_{f_2,g_2}-S_{f_1f_2, g_1g_2}$$
is compact if and only if $H^{*}_{\overline{\left(g_1-f_1\right)}}H_{f_2}$, $T_{\left(f_1-g_1\right)}H^{*}_{\overline{g_2}}$, $S_{\left(g_1-f_1\right)}H_{f_2}$ and $H_{\left(f_1-g_1\right)}H^{*}_{\overline{g_2}}$ are compact.
\end{lem}
\begin{proof}
We have by the Lemma \ref{lem1} that
\begin{align*}
&\ \ \ \ S_{f_1,g_1}S_{f_2,g_2}-S_{f_1f_2, g_1g_2}\\
&=\left (\begin{matrix}
  T_{f_1} & H_{\overline{g_1}}^* \\
    H_{f_1} & S_{g_1}
  \end{matrix}\right)\left (\begin{matrix}
    T_{f_2} & H_{\overline{g_2}}^* \\
    H_{f_2} & S_{g_2}
  \end{matrix}\right )-\left(\begin{matrix}
   T_{f_1f_2} & H_{\overline{g_1g_2}}^* \\
    H_{f_1f_2} & S_{g_1g_2}
  \end{matrix}\right)\\
  &=\left (\begin{matrix}
  T_{f_1}T_{f_2}+ H_{\overline{g_1}}^{*}H_{f_2}-T_{f_1f_2}   &  T_{f_1}H_{\overline{g_2}}^{*}+ H_{\overline{g_1}}^{*}S_{g_2}-H_{\overline{g_1g_2}}^* \\
  H_{f_1}T_{f_2}+S_{g_1}H_{f_2}-H_{f_1f_2} & H_{f_1}H_{\overline{g_2}}^{*}+S_{g_1}S_{g_2}-S_{g_1g_2}
  \end{matrix}\right)\\ &=\left (\begin{matrix}
  H_{\overline{g_1}}^{*}H_{f_2}-H_{\overline{f_1}}^{*}H_{f_2} &  T_{f_1}H_{\overline{g_2}}^{*}-T_{g_1}H_{\overline{g_2}}^{*} \\
  S_{g_1}H_{f_2}-S_{f_1}H_{f_2} & H_{f_1}H_{\overline{g_2}}^{*}-H_{g_1}H_{\overline{g_2}}^{*}
  \end{matrix}\right)\\ &=\left (\begin{matrix}
  H_{\overline{\left(g_1-f_1\right)}}^{*}H_{f_2} &  T_{\left(f_1-g_1\right)}H_{\overline{g_2}}^{*} \\
  S_{\left(g_1-f_1\right)}H_{f_2} & H_{\left(f_1-g_1\right)}H_{\overline{g_2}}^{*}
  \end{matrix}\right).
 \end{align*}
 The penultimate equality follows from
 $$T_{f_1f_2}-T_{f_1}T_{f_2}=H^{*}_{\overline{f_1}}H_{f_2},\ H^{*}_{\overline{g_1g_2}}-H^{*}_{\overline{g_1}}S_{g_2}=T_{g_1}H^{*}_{\overline{g_2}},$$
 $$H_{f_1f_2}-H_{f_1}T_{f_2}=S_{f_1}H_{f_2},\ S_{g_1g_2}-S_{g_1}S_{g_2}=H_{g_1}H^{*}_{\overline{g_2}}.$$
Therefore  $S_{f_1,g_1}S_{f_2,g_2}-S_{f_1f_2, g_1g_2}$ is compact if and only if $H^{*}_{\overline{\left(g_1-f_1\right)}}H_{f_2}$, $T_{\left(f_1-g_1\right)}H^{*}_{\overline{g_2}}$, $S_{\left(g_1-f_1\right)}H_{f_2}$ and $H_{\left(f_1-g_1\right)}H^{*}_{\overline{g_2}}$  are compact.
\end{proof}

To study the compactness of  products of Hankel and Toeplitz operators on the Hardy space, the following operator $V$ is very useful.
Define the operator  $V:\ L^2\rightarrow L^2$ by $$Vf(z)=\overline{zf(z)}, \ \ \  f\in L^2, \ z\in \partial \mathbb D.$$
It is easy to check that  $V$ is anti-unitary and moreover,
 $$V=V^{-1}=V^*$$
 on  $L^2$. For a general anti-linear operator $V$, $V^*$ is the anti-linear operator  defined via the property
 $$\overline{\langle Vf, g\rangle}=\langle f, V^*g\rangle$$
for $f$ and $g$ in $L^2$.

\begin{lem}\label{lem3}
For $f\in L^2$, then $$VP(f)=(I-P)V(f).$$
\end{lem}
\begin{proof}
For any  $f$ in $L^2$, we write $f=f_{+}+f_{-},$
where $f_{+}=Pf$ and $f_{-}=(I-P)f$. Then  we have
\begin{align*}
VP(f)(w)&=Vf_{+}(w)\\
     &=\overline{w}\overline{f_{+}(w)}\\
     &=\overline{w}\overline{f_{+}(w)}+(I-P)(\overline{w}\overline{f_{-}(w)})\\
     &=(I-P)(\overline{w}\overline{f_{+}(w)}+\overline{w}\overline{f_{-}(w)})\\
     &=(I-P)(\overline{w}\overline{f(w)})\\
     &=(I-P)V(f)(w)
\end{align*}
for each $w\in \partial\mathbb D$, to complete the proof.
\end{proof}

\begin{rem}\label{rem4}
Observe that Lemma \ref{lem3} easily leads to the following two relations:
$$VH_{\varphi}=H_{\varphi}^*V\  \ \ \  \mathrm{and} \ \ \ \ S_{\varphi}V=VT_{\overline{\varphi}},$$
 which will be used repeatedly later on.
\end{rem}

For $x$ and $y$ in $L^2$, we use $x\otimes y$ to denote the following rank-one operator: for $f\in L^2$,
$$(x\otimes y)(f)=\langle f, y\rangle x.$$

The following two lemmas about the Toeplitz and Hankel operators on $H^2$ established in \cite[Lemmas 1 and Lemma 2]{Zheng} are useful tools to study the compactness of the product of Hankel operators and compact operators in the Toeplitz algebra.
\begin{lem}\label{lem5}
Let $f$, $g$ be in $L^2$, and $z\in \D$. Then
$$H_{f}^*H_g-T_{\phi_z}^*H_{f}^*H_gT_{\phi_z}=V\big[(H_fk_z)\otimes(H_gk_z)\big]V^*.$$
\end{lem}

\noindent Here $$k_z(e^{i\theta})=\frac{\sqrt{1-|z|^2}}{1-\overline{z}e^{i\theta}}$$ is the normalized reproducing kernel for the Hardy space, and $\phi_z$ denotes  the M\"{o}bius map:
$$\phi_z(w)=\frac{z-w}{1-\overline{z}w} \ \ \ \ \  \big(z, w\in \mathbb D\big).$$

\begin{lem}\label{lem6}
Let $K$ be a compact operator on $H^2$. Then we have
$$\lim_{|z|\rightarrow 1^-} \|K-T_{\phi_z}^*KT_{\phi_z}\|=0.$$
\end{lem}
We also need the following lemmas, which are introduced in \cite[Lemma 3.4 and Lemma 3.5]{WZZ}.
\begin{lem}\label{lem7}
Let $f$, $g$ be in $L^2$, and $z\in \D$. Then
$$H_fT_gT_{\phi_z}-S_{\phi_z}H_fT_g=H_fk_z\otimes T_{\overline{g\phi_z}}k_z=-(H_{f}k_z)\otimes(VH_{g}k_z)$$
for all $z\in \mathbb D$.
\end{lem}

\begin{lem}\label{lem8}
Let $K:H^2\rightarrow \overline{zH^2}$ be a compact operator. Then
$$\lim_{|z|\rightarrow 1^-}\|S_{\phi_z}K-KT_{\phi_z}\|=0.$$
\end{lem}

As in \cite{Gar}, a Douglas algebra is, by definition, a closed subalgebra of $L^\infty$ which contains $H^\infty$. As Douglas algebras play a prominent role in various problems on Toeplitz and Hankel operators, we need to review some important properties of them. Observe that $H^\infty$ is a commutative Banach algebra, we can identify the maximal ideal space $\mathcal{M}(H^\infty)$ as the set of multiplicative linear functionals on $H^\infty$. Endowed with the weak star topology it inherits as a subset of the dual space of $H^\infty$, $\mathcal{M}(H^\infty)$ is a compact Hausdorff space. Identifying a point in the open unit disk $\D$ with the functional of evaluation at this point, we may regard the disk $\D$ as a subset of $\mathcal{M}(H^\infty)$. Using the Gelfand transform we regard every function in $H^\infty$ as a continuous function on $\mathcal{M}(H^\infty)$. The deepest result concerning $\mathcal{M}(H^\infty)$ is the famous corona theorem of Carleson, stating that $\D$ is dense in $\mathcal{M}(H^\infty)$ under the weak star topology (for details, see \cite{Dur} and \cite{Gar}).

It is a consequence of the Gleason-Whitney theorem that the maximal ideal space of a Douglas algebra $B$ is a naturally imbedded in $\mathcal{M}(H^\infty)$. Thus we may identify the maximal ideal space $\mathcal{M}(H^\infty+C)$ of the Sarason algebra $H^\infty+C$ with a subset of $\mathcal{M}(H^\infty)$, where $C$ is the algebra of continuous functions on $\partial\mathbb D$. A subset of $\mathcal{M}(L^\infty)$ will be a support set if it is the (closed) support of the representing measure for a functional in $\mathcal{M}(H^\infty+C)$, see \cite{Gar} and \cite{Hof} for more details. Let $m$ be in $\mathcal{M}(H^\infty+C)$ and let $d\mu_m$ denote the unique representing measure for $m$ with support $S_m$, i.e.,\\
(1)~ for all $f$ and $g$ in $H^\infty$, $$m(fg)=\int_{S_m}fg~d\mu_m=\bigg(\int_{S_m} f d\mu_m\bigg)\bigg(\int_{S_m} g d\mu_m\bigg); $$
(2)~ if $h\geqslant 0$ a.e. in $L^1(d\mu_m)$ such that $$\int_{S_m}fh~d\mu_m=\int_{S_m} f d\mu_m$$ for all $f\in H^\infty$, then we have $h=1$ a.e. $d\mu_m$.

Suppose that $m\in \mathcal{M}(H^\infty+C)$ and $z\mapsto \xi_z$ is a mapping from the unit disk $\mathbb D$ into some topological space $X$. Let $\eta$ be in $X$. We use the notation
$$\lim_{z\rightarrow m} \xi_z=\eta$$
to denote that for each open set $\mathcal{U}(\eta)\subset X$ containing $\eta$, there exists an open subset $\mathcal{O}(m)$ of $\mathcal{M}(H^\infty+C)$  containing $m$ such that $\xi_z\in \mathcal{U}$ for all $z\in \mathcal{O}(m)\cap \mathbb D$.

For a function $F$ on the disk $\D$ and $m$ in $\mathcal{M}(H^\infty+C)$, we say
$$\lim_{z\rightarrow m} F(z)=0$$
if for every net $\{z_\alpha\}\subset \D$ converging to $m$,
$$\lim_{z_\alpha\rightarrow m} F(z_\alpha)=0.$$
We shall emphasize here that we  deal with nets rather than sequences since the the topology of $\mathcal{M}(H^\infty+C)$ is not metrizable.

With the above notations and concepts about $H^2$ theory on a support set, we quote the following lemma  obtained in \cite[Lemmas 2.5 and 2.6]{GpZ}.
\begin{lem}\label{lem9}
Let $f$ be in $L^\infty$ and $m\in \mathcal M(H^\infty+C)$. Denote the support set for $m$ by $S_m$. Then
the following three conditions are equivalent:\\
$(1)$ $f|_{S_m}\in H^\infty|_{S_m}$;\vspace{1.5mm}\\
$(2)$ $\lim\limits_{z\rightarrow m}\|H_{f}k_z\|_2=0$;\vspace{1.5mm}\\
$(3)$ $\varliminf\limits_{z\rightarrow m}\|H_{f}k_z\|_2=0$.
\end{lem}

\begin{rem}\label{rem10}
The Carleson-Corona theorem (\cite{Gar}) tells us that the conclusions of Lemmas \ref{lem6} and Lemma \ref{lem8} are equivalent to the condition that for each $m\in \mathcal M(H^\infty+C)$,
$$\lim\limits_{z\rightarrow m}\|K-T_{\phi_z}^*KT_{\phi_z}\|=0\ \ \ \ \ \mathrm{and} \ \ \ \ \ \lim_{z\rightarrow m}\|S_{\phi_z}K-KT_{\phi_z}\|=0$$
for $z$ in the unit disk $\mathbb D$.
\end{rem}

\section{The proof of Theorem \ref{MR}}
The following lemma will be needed in the proof of main theorem, which was established in \cite[Lemma 17]{GkZ1}.
\begin{lem}\label{lem12}
Suppose that $\varphi$ and $\psi$ are in $L^\infty$. Let $m\in \mathcal M(H^\infty+C)$. If
$$\lim\limits_{z\rightarrow m}\|H_{\varphi}k_z\|_2=0,$$
then we have
$$\lim\limits_{z\rightarrow m}\|H_{\varphi}T_{\psi}k_z\|_2=0.$$
\end{lem}

We will show the Theorem \ref{MR} with the following two propositions.
\begin{prop}\label{prop14}
Let $f_1, f_2, g_1,g_2\in L^{\infty}$. For each support set $S$, if one of the following holds:\\
$(1)$ $\left(f_1-g_1\right)|_{S}=0$;\\
$(2)$ $f_2|_{S}$ and $\overline{g_2}|_{S}$ are in $H^{\infty}|_{S}$.\\
Then the semi-commutator $S_{f_1,g_1}S_{f_2,g_2}-S_{f_1f_2, g_1g_2}$ is compact.
\end{prop}
\begin{proof}
For each support set $S_m$ of $m\in \mathcal{M}(H^\infty+C)$, we suppose that $\left(f_1-g_1\right)|_{S_m}=0$ or $f_{2}|_{S_m}$ and $\overline{g_{2}}|_{S_m}$ are in $H^{\infty}|_{S_m}$. Then we have by Lemma \ref{lem9} that
\begin{equation}\label{eq6}
\begin{aligned}
\lim\limits_{z\rightarrow m}\|H_{\left(f_1-g_1\right)}k_z\|_2&=\lim\limits_{z\rightarrow m}\|H_{\overline{\left(f_1-g_1\right)}}k_z\|_2=0,\\
\lim\limits_{z\rightarrow m}\|H_{\left(f_1-g_1\right)f_2}k_z\|_2&=\lim\limits_{z\rightarrow m}\|H_{\overline{\left(f_1-g_1\right)}f_2}k_z\|_2=0,\\
\lim\limits_{z\rightarrow m}\|H_{\left(f_1-g_1\right)\overline{g_2}}k_z\|_2&=\lim\limits_{z\rightarrow m}\|H_{\overline{\left(f_1-g_1\right)g_2}}k_z\|_2=0
\end{aligned}
\end{equation}
or
\begin{equation}\label{eq7}
\lim\limits_{z\rightarrow m}\|H_{f_2}k_z\|_2=\lim\limits_{z\rightarrow m}\|H_{\overline{g_{2}}}k_z\|_2=0.
\end{equation}
According to Lemma \ref{lem2}, we are going to prove $H^{*}_{\overline{\left(g_1-f_1\right)}}H_{f_2}$, $T_{\left(f_1-g_1\right)}H^{*}_{\overline{g_2}}$, $S_{\left(g_1-f_1\right)}H_{f_2}$ and $H_{\left(f_1-g_1\right)}H^{*}_{\overline{g_2}}$ are compact.\\
By Lemma \ref{lem5}, we have
\begin{equation}\label{eq3}
\begin{aligned}
&\~~~\~~~\lim\limits_{z\rightarrow m}\left\|H^{*}_{\overline{g_1-f_1}}H_{f_2}-T_{\phi_z}^*H^{*}_{\overline{g_1-f_1}}H_{f_2}T_{\phi_z}\right\|\\
&=\lim\limits_{z\rightarrow m}\left\| V\left\{\left[H_{\overline{\left(g_1-f_1\right)}}k_z\right])\otimes\left(H_{f_2}k_z\right)\right\}V^*\right\|\\
&=\lim\limits_{z\rightarrow m}\left\|\left[H_{\overline{\left(g_1-f_1\right)}}k_z\right]\otimes(H_{f_2}k_z)\right\|\\
&=\lim\limits_{z\rightarrow m}\left\|H_{\overline{\left(g_1-f_1\right)}}k_z\right\|_2\cdot\left\|H_{f_2}k_z\right\|_2\\
&=0.
\end{aligned}
\end{equation}
On the other hand, noting
$$H^{*}_{\overline{g_1-f_1}}H_{f_2}=T_{\left(\overline{g_1-f_1}\right)f_{2}}-T_{\overline{g_1-f_1}}T_{f_{2}},$$
which is a finite sum of finite products of Toeplitz operators. According to \cite[Theorem 12]{GkZ2}, we obtain by (\ref{eq3}) that $H^{*}_{\overline{\left(g_1-f_1\right)}}H_{f_2}$ is equal to a compact perturbation of a Toeplitz operator, i.e.,
$$H^{*}_{\overline{\left(g_1-f_1\right)}}H_{f_2}=T_{\left(g_1-f_1\right)f_{2}}-T_{\left(g_1-f_1\right)}T_{f_{2}}
=T_h+K$$
for some $h\in L^\infty$ and some compact operator $K$.  Thus 
$$K=T_{\left(g_1-f_1\right)f_{2}}-T_{\left(g_1-f_1\right)}T_{f_{2}}-T_h$$ 
belongs to the Toeplitz algebra $\mathscr{T}_{L^{\infty}}$. We conclude by \cite[Corollary 6]{BH} that $h=0$ a.e., which implies that $H^{*}_{\overline{\left(g_1-f_1\right)}}H_{f_2}=K$ is compact.

To show the operator $H_{\left(f_1-g_1\right)}H^{*}_{\overline{g_2}}$ is compact, we recall that 
$$VH_{\varphi}=H_{\varphi}^*V.$$
Then
\begin{equation}\label{eq4}
\begin{array}{l}
V\left[H_{\left(f_1-g_1\right)}H^{*}_{\overline{g_2}}\right]V =H_{\left(f_1-g_1\right)}^{*}VVH_{\overline{g_2}}
=H_{\left(f_1-g_1\right)}^{*}H_{\overline{g_2}},
\end{array}
\end{equation}
where the second equality follows from $V^2=I$. Using the same method as the above, we can show  similarly that  $H_{\left(f_1-g_1\right)}^{*}H_{\overline{g_2}}$ is compact. Furthermore, (\ref{eq4}) gives us that
$H_{\left(f_1-g_1\right)}H^{*}_{\overline{g_2}}$ is also compact.

Now we turn to the proof of the compactness of the operator $T_{\left(f_1-g_1\right)}H^{*}_{\overline{g_2}}$. We need only to consider the compactness of the following operator
\begin{align*}
\left[T_{\left(f_1-g_1\right)}H^{*}_{\overline{g_2}}\right]\left[T_{\left(f_1-g_1\right)}H^{*}_{\overline{g_2}}\right]^{*}&=T_{\left(f_1-g_1\right)}H^{*}_{\overline{g_2}}H_{\overline{g_2}}T_{\overline{\left(f_1-g_1\right)}}\\
&=T_{\left(f_1-g_1\right)}\left(T_{g_2\overline{g_2}}-T_{g_2}T_{\overline{g_2}}\right)T_{\overline{\left(f_1-g_1\right)}},
\end{align*}
which is in the Toeplitz algebra $\mathscr{T}_{L^{\infty}}$ and the symbol map maps it to 0. Using \cite[Theorem 12]{GkZ2} again, we need only to prove that  
$$\lim\limits_{z\rightarrow m} \left\|\left[T_{\left(f_1-g_1\right)}H^{*}_{\overline{g_2}}\right]\left[T_{\left(f_1-g_1\right)}H^{*}_{\overline{g_2}}\right]^{*}-T_{\phi_z}^*\left[T_{\left(f_1-g_1\right)}H^{*}_{\overline{g_2}}\right]\left[T_{\left(f_1-g_1\right)}H^{*}_{\overline{g_2}}\right]^{*}T_{\phi_z}\right\|=0.$$
Since
\begin{align*}\label{eq5}
&\~~\~~\~~\ T_{\phi_z}^*\left[T_{\left(f_1-g_1\right)}H^{*}_{\overline{g_2}}\right]\left[T_{\left(f_1-g_1\right)}H^{*}_{\overline{g_2}}\right]^{*}T_{\phi_z}\\
&=\left[H_{\overline{g_2}}T_{\overline{\left(f_1-g_1\right)}}T_{\phi_{z}}\right]^{*}\left[H_{\overline{g_2}}T_{\overline{\left(f_1-g_1\right)}}T_{\phi_{z}}\right]\\
&=\left\{S_{\phi_{z}}H_{\overline{g_2}}T_{\overline{\left(f_1-g_1\right)}}-\left(H_{\overline{g_2}}k_{z}\right)\otimes\left[VH_{\overline{\left(f_1-g_1\right)}}k_{z}\right]\right\}^{*}\\
&\~~\~~\~~\ \cdot\left\{S_{\phi_{z}}H_{\overline{g_2}}T_{\overline{\left(f_1-g_1\right)}}-\left(H_{\overline{g_2}}k_{z}\right)\otimes\left[VH_{\overline{\left(f_1-g_1\right)}}k_{z}\right]\right\} \~~~~\~~~(\text{by\ Lemma\ \ref{lem7}})\\
&=T_{\left(f_1-g_1\right)}H^{*}_{\overline{g_2}}S_{\phi_{z}}^{*}S_{\phi_{z}}H_{\overline{g_2}}T_{\overline{\left(f_1-g_1\right)}}
-\left[T_{\left(f_1-g_1\right)}H^{*}_{\overline{g_2}}S_{\phi_{z}}^{*}H_{\overline{g_2}}k_{z}\right]\otimes\left[VH_{\overline{\left(f_1-g_1\right)}}k_{z}\right]\\
&\~~\~~\~~\ -\left[VH_{\overline{\left(f_1-g_1\right)}}k_{z}\right]\otimes\left[T_{\left(f_1-g_1\right)}H^{*}_{\overline{g_2}}S_{\phi_{z}}^{*}H_{\overline{g_2}}k_{z}\right]\\
&\~~\~~\~~\ +\left\{\left(H_{\overline{g_2}}k_{z}\right)\otimes\left[VH_{\overline{\left(f_1-g_1\right)}}k_{z}\right]\right\}^{*}\left\{\left(H_{\overline{g_2}}k_{z}\right)\otimes\left[VH_{\overline{\left(f_1-g_1\right)}}k_{z}\right]\right\}\\
&=T_{\left(f_1-g_1\right)}H^{*}_{\overline{g_2}}\left(I-Vk_z\otimes Vk_z\right)H_{\overline{g_2}}T_{\overline{\left(f_1-g_1\right)}}\\
&\~~\~~\~~\ -\left[T_{\left(f_1-g_1\right)}H^{*}_{\overline{g_2}}S_{\phi_{z}}^{*}H_{\overline{g_2}}k_{z}\right]\otimes\left[VH_{\overline{\left(f_1-g_1\right)}}k_{z}\right]\\
&\~~\~~\~~\ -\left[VH_{\overline{\left(f_1-g_1\right)}}k_{z}\right]\otimes\left[T_{\left(f_1-g_1\right)}H^{*}_{\overline{g_2}}S_{\phi_{z}}^{*}H_{\overline{g_2}}k_{z}\right]\\
&\~~\~~\~~\ +\left\{\left(H_{\overline{g_2}}k_{z}\right)\otimes\left[VH_{\overline{\left(f_1-g_1\right)}}k_{z}\right]\right\}^{*}\left\{\left(H_{\overline{g_2}}k_{z}\right)\otimes\left[VH_{\overline{\left(f_1-g_1\right)}}k_{z}\right]\right\}\\
&=\left[T_{\left(f_1-g_1\right)}H^{*}_{\overline{g_2}}\right]\left[T_{\left(f_1-g_1\right)}H^{*}_{\overline{g_2}}\right]^{*}
-\left[T_{\left(f_1-g_1\right)}H^{*}_{\overline{g_2}}Vk_z\right]\otimes \left[T_{\left(f_1-g_1\right)}H^{*}_{\overline{g_2}}Vk_z\right]\\
&\~~\~~\~~\ -\left[T_{\left(f_1-g_1\right)}H^{*}_{\overline{g_2}}S_{\phi_{z}}^{*}H_{\overline{g_2}}k_{z}\right]\otimes\left[VH_{\overline{\left(f_1-g_1\right)}}k_{z}\right]\\
&\~~\~~\~~\ -\left[VH_{\overline{\left(f_1-g_1\right)}}k_{z}\right]\otimes\left[T_{\left(f_1-g_1\right)}H^{*}_{\overline{g_2}}S_{\phi_{z}}^{*}H_{\overline{g_2}}k_{z}\right]\\
&\~~\~~\~~\ +\left\{\left(H_{\overline{g_2}}k_{z}\right)\otimes\left[VH_{\overline{\left(f_1-g_1\right)}}k_{z}\right]\right\}^{*}\left\{\left(H_{\overline{g_2}}k_{z}\right)\otimes\left[VH_{\overline{\left(f_1-g_1\right)}}k_{z}\right]\right\}\\
&=\left[T_{\left(f_1-g_1\right)}H^{*}_{\overline{g_2}}\right]\left[T_{\left(f_1-g_1\right)}H^{*}_{\overline{g_2}}\right]^{*}
-\left[VS_{\overline{\left(f_1-g_1\right)}}H_{\overline{g_2}}k_z\right]\otimes \left[VS_{\overline{\left(f_1-g_1\right)}}H_{\overline{g_2}}k_z\right]\\
&\~~\~~\~~\ -\left[T_{\left(f_1-g_1\right)}H^{*}_{\overline{g_2}}S_{\phi_{z}}^{*}H_{\overline{g_2}}k_{z}\right]\otimes\left[VH_{\overline{\left(f_1-g_1\right)}}k_{z}\right]\\
&\~~\~~\~~\ -\left[VH_{\overline{\left(f_1-g_1\right)}}k_{z}\right]\otimes\left[T_{\left(f_1-g_1\right)}H^{*}_{\overline{g_2}}S_{\phi_{z}}^{*}H_{\overline{g_2}}k_{z}\right]\\
&\~~\~~\~~\ +\left\{\left(H_{\overline{g_2}}k_{z}\right)\otimes\left[VH_{\overline{\left(f_1-g_1\right)}}k_{z}\right]\right\}^{*}\left\{\left(H_{\overline{g_2}}k_{z}\right)\otimes\left[VH_{\overline{\left(f_1-g_1\right)}}k_{z}\right]\right\},
\end{align*}
we have 
\begin{equation}\label{eq8}
\begin{aligned}
&\~~~\~~~ \left[T_{\left(f_1-g_1\right)}H^{*}_{\overline{g_2}}\right]\left[T_{\left(f_1-g_1\right)}H^{*}_{\overline{g_2}}\right]^{*}-T_{\phi_z}^*\left[T_{\left(f_1-g_1\right)}H^{*}_{\overline{g_2}}\right]\left[T_{\left(f_1-g_1\right)}H^{*}_{\overline{g_2}}\right]^{*}T_{\phi_z}\\
&=\left[VS_{\overline{\left(f_1-g_1\right)}}H_{\overline{g_2}}k_z\right]\otimes \left[VS_{\overline{\left(f_1-g_1\right)}}H_{\overline{g_2}}k_z\right]\\
&\~~\~~\~~\ +\left[T_{\left(f_1-g_1\right)}H^{*}_{\overline{g_2}}S_{\phi_{z}}^{*}H_{\overline{g_2}}k_{z}\right]\otimes\left[VH_{\overline{\left(f_1-g_1\right)}}k_{z}\right]\\
&\~~\~~\~~\ +\left[VH_{\overline{\left(f_1-g_1\right)}}k_{z}\right]\otimes\left[T_{\left(f_1-g_1\right)}H^{*}_{\overline{g_2}}S_{\phi_{z}}^{*}H_{\overline{g_2}}k_{z}\right]\\
&\~~\~~\~~\ -\left\{\left(H_{\overline{g_2}}k_{z}\right)\otimes\left[VH_{\overline{\left(f_1-g_1\right)}}k_{z}\right]\right\}^{*}\left\{\left(H_{\overline{g_2}}k_{z}\right)\otimes\left[VH_{\overline{\left(f_1-g_1\right)}}k_{z}\right]\right\}.
\end{aligned}
\end{equation}
Thus we only need to show 
$$\lim\limits_{z\rightarrow m}\left\|\left[VS_{\overline{\left(f_1-g_1\right)}}H_{\overline{g_2}}k_z\right]\otimes \left[VS_{\overline{\left(f_1-g_1\right)}}H_{\overline{g_2}}k_z\right]\right\|=\lim\limits_{z\rightarrow m}\left\|S_{\overline{\left(f_1-g_1\right)}}H_{\overline{g_2}}k_z\right\|_2^{2}=0.$$
If (\ref{eq7}) holds, we naturally have $\lim\limits_{z\rightarrow m}\left\|S_{\overline{f_1-g_1}}H_{\overline{g_2}}k_z\right\|_2=0$ since $f_{1}, g_{1}\in L^{\infty}$. Else if (\ref{eq6}) holds, then
\begin{align*}\label{eq9}
S_{\overline{f_1-g_1}}H_{\overline{g_2}}k_z
        &= (I-P)\left[\overline{\left(f_1-g_1\right)}(I-P)(\overline{g_2}k_z)\right]\\
        &= (I-P)\left[\overline{\left(f_1-g_1\right)}\cdot\left(\overline{g_2}k_z\right)-\overline{\left(f_1-g_1\right)}P(\overline{g_2}k_z)\right]\\
        &= H_{\overline{\left(f_1-g_1\right)g_2}}k_z-H_{\overline{\left(f_1-g_1\right)}}T_{\overline{g_2}}k_z,
\end{align*}
which implies $\lim\limits_{z\rightarrow m}\left\|S_{\overline{f_1-g_1}}H_{\overline{g_2}}k_z\right\|_2=0$ 
by Lemma \ref{lem12}. It gives that $\left[T_{\left(f_1-g_1\right)}H^{*}_{\overline{g_2}}\right]\left[T_{\left(f_1-g_1\right)}H^{*}_{\overline{g_2}}\right]^{*}$ is compact, so $T_{\left(f_1-g_1\right)}H^{*}_{\overline{g_2}}$ is also compact.

Using anti-unitary operator $V$ again, we have
 $$VS_{\left(g_1-f_1\right)}H_{f_2}V=T_{\overline{g_1-f_1}}V^2H_{f_2}^{*}=T_{\overline{g_1-f_1}}H_{f_2}^{*},$$
 which has the same form as the compact operator $T_{\left(f_1-g_1\right)}H^{*}_{\overline{g_2}}$. Thus we can conclude that $T_{\overline{g_1-f_1}}H_{f_2}^{*}$ is compact by the same argument as above. It follows that 
$S_{\left(g_1-f_1\right)}H_{f_2}$ is compact, to complete this proof. 
\end{proof}

\begin{prop}\label{prop11}
Let $f_1, f_2, g_1,g_2\in L^{\infty}$. Then the semi-commutator $S_{f_1,g_1}S_{f_2,g_2}-S_{f_1f_2, g_1g_2}$ is compact, then for each support set $S$, one of the following holds:\\
$(1)$ $\left(f_1-g_1\right)|_{S}=0$;\\
$(2)$ $f_2|_{S}$ and $\overline{g_2}|_{S}$ are in $H^{\infty}|_{S}$.
\end{prop}
\begin{proof}
We suppose that $S_{f_1,g_1}S_{f_2,g_2}-S_{f_1f_2, g_1g_2}$ is compact, where $f_1, f_2, g_1,g_2$ are in $L^{\infty}$. Then we have by Lemma \ref{lem2} that 
$$H^{*}_{\overline{\left(g_1-f_1\right)}}H_{f_2}, T_{\left(f_1-g_1\right)}H^{*}_{\overline{g_2}}, S_{\left(g_1-f_1\right)}H_{f_2}\  \mathrm{and}\ H_{\left(f_1-g_1\right)}H^{*}_{\overline{g_2}}$$
are compact.

Considering the compact operator $H^{*}_{\overline{g_1-f_1}}H_{f_2}$, we have by Lemma \ref{lem5} and Lemma \ref{lem6} that
$$H^{*}_{\overline{\left(g_1-f_1\right)}}H_{f_2}-T_{\phi_z}^*H^{*}_{\overline{\left(g_1-f_1\right)}}H_{f_2}T_{\phi_z}=V\left\{\left[H_{\overline{\left(g_1-f_1\right)}}k_z\right]\otimes(H_{f_2}k_z)\right\}V^*$$
and for each $m\in \mathcal M(H^\infty+C)$,
\begin{equation}\label{eq1}
\begin{aligned}
&\~~~\~~~\lim\limits_{z\rightarrow m}\left\| V\left\{\left[H_{\overline{\left(g_1-f_1\right)}}k_z\right]\otimes\left(H_{f_2}k_z\right)\right\}V^*\right\|\\
&=\lim\limits_{z\rightarrow m}\left\|\left[H_{\overline{\left(g_1-f_1\right)}}k_z\right]\otimes\left(H_{f_2}k_z\right)\right\|\\
&=\lim\limits_{z\rightarrow m}\left\|H_{\overline{\left(g_1-f_1\right)}}k_z\right\|_2\cdot\left\|H_{f_2}k_z\right\|_2\\
&=0,
\end{aligned}
\end{equation}
which gives that $\lim\limits_{z\rightarrow m}\left\|H_{\overline{\left(g_1-f_1\right)}}k_z\right\|_2=0$ or $\lim\limits_{z\rightarrow m}\left\|H_{f_2}k_z\right\|_2=0$. By Lemma \ref{lem9}, we conclude that $\overline{\left(g_1-f_1\right)}|_{S_m}\in H^{\infty}|_{S_m}$ or $f_{2}|_{S_m}\in H^{\infty}|_{S_m}$. \\
Since 
$$H_{\left(f_1-g_1\right)}H^{*}_{\overline{g_2}}=H_{\left(f_1-g_1\right)}V^{2}H^{*}_{\overline{g_2}}=VH^{*}_{\left(f_1-g_1\right)}H_{\overline{g_2}}V,$$
we also have  $\left(f_1-g_1\right)|_{S_m}\in H^{\infty}|_{S_m}$ or $\overline{g_{2}}|_{S_m}\in H^{\infty}|_{S_m}$
by the same argument.

Next we consider the compact operator $T_{\left(f_1-g_1\right)}H^{*}_{\overline{g_2}}$, and we know that the duality of $T_{\left(f_1-g_1\right)}H^{*}_{\overline{g_2}}$ is also compact. It follows that $H_{\overline{g_2}}T_{\overline{\left(f_1-g_1\right)}}$ is compact. Thus we conclude by Lemma \ref{lem7} and Lemma \ref{lem8} that 
$$H_{\overline{g_2}}T_{\overline{\left(f_1-g_1\right)}}T_{\phi_z}-S_{\phi_z}H_{\overline{g_2}}T_{\overline{\left(f_1-g_1\right)}}=H_{\overline{g_2}}k_z\otimes T_{\left(f_1-g_1\right)\overline{\phi_z}}k_z=-(H_{\overline{g_2}}k_z)\otimes(VH_{\overline{\left(f_1-g_1\right)}}k_z)$$
and
\begin{equation}\label{eq2}
\begin{aligned}
\lim_{z\rightarrow m}\left\|-(H_{\overline{g_2}}k_z)\otimes\left[VH_{\overline{\left(f_1-g_1\right)}}k_z\right]\right\|
&=\lim\limits_{z\rightarrow m}\left\|H_{\overline{g_2}}k_z\right\|_2\cdot\left\|VH_{\overline{\left(f_1-g_1\right)}}k_z\right\|_2\\
&=\lim\limits_{z\rightarrow m}\left\|H_{\overline{g_2}}k_z\right\|_2\cdot\left\|H_{\overline{\left(f_1-g_1\right)}}k_z\right\|_2\\
&=0,
\end{aligned}
\end{equation}
which induces that $\lim\limits_{z\rightarrow m}\left\|H_{\overline{g_2}}k_z\right\|_2=0$ or $\lim\limits_{z\rightarrow m}\left\|H_{\overline{\left(f_1-g_1\right)}}k_z\right\|_2=0$, and we have by Lemma \ref{lem9} that $\overline{g_2}|_{S_m}\in H^{\infty}|_{S_m}$ or $\overline{\left(f_1-g_1\right)}|_{S_m}\in H^{\infty}|_{S_m}$. Using the same argument as above, we conclude that $f_2|_{S_m}\in H^{\infty}|_{S_m}$ or $\left(g_1-f_1\right)|_{S_m}\in H^{\infty}|_{S_m}$
because $T_{\overline{g_1-f_1}}H^{*}_{f_2}=VS_{g_1-f_1}H_{f_2}V$ is compact.

To finish the proof of this proposition, we are going to sort out above results. And we conclude that if $S_{f_1,g_1}S_{f_2,g_2}-S_{f_1f_2, g_1g_2}$ is compact, then for each support set $S_m$ of $m\in \mathcal{M}(H^\infty+C)$, 
the following conditions hold:\\
$(\rmnum{1})$ $\overline{\left(g_1-f_1\right)}|_{S_m}\in H^{\infty}|_{S_m}$ or $f_{2}|_{S_m}\in H^{\infty}|_{S_m}$;\\
$(\rmnum{2})$ $\left(f_1-g_1\right)|_{S_m}\in H^{\infty}|_{S_m}$ or $\overline{g_{2}}|_{S_m}\in H^{\infty}|_{S_m}$;\\
$(\rmnum{3})$ $\overline{g_2}|_{S_m}\in H^{\infty}|_{S_m}$ or $\overline{\left(f_1-g_1\right)}|_{S_m}\in H^{\infty}|_{S_m}$;\\
$(\rmnum{4})$ $f_2|_{S_m}\in H^{\infty}|_{S_m}$ or $\left(g_1-f_1\right)|_{S_m}\in H^{\infty}|_{S_m}$.\\
It gives that $\left(f_1-g_1\right)|_{S_m}$ is a constant or $f_{2}|_{S_m}, \overline{g_{2}}|_{S_m}\in H^{\infty}|_{S_m}$.

Finally, we will show  $\left(f_1-g_1\right)|_{S_m}=0$ if  $\left(f_1-g_1\right)|_{S_m}=c$ for some constant $c$, and $f_{2}|_{S_m}$ or $\overline{g_{2}}|_{S_m}$ is not in $H^{\infty}|_{S_m}$. We consider the following two cases:
\begin{enumerate}
    \item $\left(f_1-g_1\right)|_{S_m}=c$ and $\overline{g_{2}}|_{S_m}\notin H^{\infty}|_{S_m}$;
    \item $\left(f_1-g_1\right)|_{S_m}=c$ and $f_2|_{S_m}\notin H^{\infty}|_{S_m}$.
\end{enumerate}

\textbf{Case 1.}  If $\left(f_1-g_1\right)|_{S_m}=c$ and $\overline{g_{2}}|_{S_m}\notin H^{\infty}|_{S_m}$, then we have
\begin{equation}\label{eq10}
\begin{aligned}
\lim\limits_{z\rightarrow m}\|H_{\left(f_1-g_1\right)}k_z\|_2&=\lim\limits_{z\rightarrow m}\|H_{\overline{\left(f_1-g_1\right)}}k_z\|_2=0.
\end{aligned}
\end{equation}

Considering equality (\ref{eq8}) in the proof of Proposition \ref{prop14}, we have 

\begin{equation}\label{eq11}
\begin{aligned}
&\~~~\~~~ \left\|\left[VS_{\overline{\left(f_1-g_1\right)}}H_{\overline{g_2}}k_z\right]\otimes \left[VS_{\overline{\left(f_1-g_1\right)}}H_{\overline{g_2}}k_z\right]\right\|\\
&\leq\left\|\left[T_{\left(f_1-g_1\right)}H^{*}_{\overline{g_2}}\right]\left[T_{\left(f_1-g_1\right)}H^{*}_{\overline{g_2}}\right]^{*}-T_{\phi_z}^*\left[T_{\left(f_1-g_1\right)}H^{*}_{\overline{g_2}}\right]\left[T_{\left(f_1-g_1\right)}H^{*}_{\overline{g_2}}\right]^{*}T_{\phi_z}\right\|\\
&\~~\~~\~~\ +\left\|\left[T_{\left(f_1-g_1\right)}H^{*}_{\overline{g_2}}S_{\phi_{z}}^{*}H_{\overline{g_2}}k_{z}\right]\otimes\left[VH_{\overline{\left(f_1-g_1\right)}}k_{z}\right]\right\|\\
&\~~\~~\~~\ +\left\|\left[VH_{\overline{\left(f_1-g_1\right)}}k_{z}\right]\otimes\left[T_{\left(f_1-g_1\right)}H^{*}_{\overline{g_2}}S_{\phi_{z}}^{*}H_{\overline{g_2}}k_{z}\right]\right\|\\
&\~~\~~\~~\ +\left\|\left\{\left(H_{\overline{g_2}}k_{z}\right)\otimes\left[VH_{\overline{\left(f_1-g_1\right)}}k_{z}\right]\right\}^{*}\left\{\left(H_{\overline{g_2}}k_{z}\right)\otimes\left[VH_{\overline{\left(f_1-g_1\right)}}k_{z}\right]\right\}\right\|.
\end{aligned}
\end{equation}
Since $T_{\left(f_1-g_1\right)}H^{*}_{\overline{g_2}}$ is compact, we have by Lemma \ref{lem6} that
$$\lim\limits_{z\rightarrow m}\left\|\left[T_{\left(f_1-g_1\right)}H^{*}_{\overline{g_2}}\right]\left[T_{\left(f_1-g_1\right)}H^{*}_{\overline{g_2}}\right]^{*}-T_{\phi_z}^*\left[T_{\left(f_1-g_1\right)}H^{*}_{\overline{g_2}}\right]\left[T_{\left(f_1-g_1\right)}H^{*}_{\overline{g_2}}\right]^{*}T_{\phi_z}\right\|=0.$$
Thus we can conclude that the right side of inequality (\ref{eq11}) tends to 0 as $z\rightarrow m$, which implies that
$$\lim\limits_{z\rightarrow m}\left\|S_{\overline{\left(f_1-g_1\right)}}H_{\overline{g_2}}k_z\right\|=0.$$
On the other hand, we see
\begin{align*}\label{eq12}
\left\|H_{\overline{\left(f_1-g_1\right)g_2}}k_z\right\|
&=\left\|S_{\overline{f_1-g_1}}H_{\overline{g_2}}k_z+H_{\overline{\left(f_1-g_1\right)}}T_{\overline{g_2}}k_z\right\|\\
&\leq\left\|S_{\overline{f_1-g_1}}H_{\overline{g_2}}k_z\right\|+\left\|H_{\overline{\left(f_1-g_1\right)}}T_{\overline{g_2}}k_z\right\|,
\end{align*}
and $\lim\limits_{z\rightarrow m}\left\|H_{\overline{\left(f_1-g_1\right)}}T_{\overline{g_2}}k_z\right\|=0$ by Lemma \ref{lem12}. It follows that $\lim\limits_{z\rightarrow m}\left\|H_{\overline{\left(f_1-g_1\right)g_2}}k_z\right\|=0$ and
$$\left[\overline{\left(f_1-g_1\right)g_2}\right]\Big|_{S_m}\in H^{\infty}|_{S_m}.$$
If $c\neq 0$, then 
$$\overline{g_2}|_{S_m=}\frac{\left[\overline{\left(f_1-g_1\right)g_2}\right]\Big|_{S_m}}{\overline{c}}\in H^{\infty}|_{S_m},$$
which is contradicted with our assumption that $\overline{g_{2}}|_{S_m}\notin H^{\infty}|_{S_m}$. It gives that $c=0$.

\textbf{Case 2.}  If $\left(f_1-g_1\right)|_{S_m}=c$ and $f_{2}|_{S_m}\notin H^{\infty}|_{S_m}$, then we consider the compact operator $T_{\overline{g_1-f_1}}H_{f_2}^{*}=VS_{\left(g_1-f_1\right)}H_{f_2}V$. Using the same argument, we conclude that 
$$\lim\limits_{z\rightarrow m}\left\|H_{\left(g_1-f_1\right)f_2}k_z\right\|=0$$ 
and $\left[\left(g_1-f_1\right)f_2\right]\big|_{S_m}\in H^{\infty}|_{S_m}$. If $c\neq 0$, then
 $$f_{2}|_{S_m}=\frac{\left[\left(g_1-f_1\right)f_2\right]\big|_{S_m}}{-c}\in H^{\infty}|_{S_m},$$
 which is contradicted with our assumption that $f_{2}|_{S_m}\notin H^{\infty}|_{S_m}$. It follows that $c=0$, to complete the proof.
\end{proof}

\section{Essentially normal Cauchy singular integral operator}
In this section, we will characterize essentially normality of $S_{f, g}$.
\begin{thm}\label{thm16}
Let $f, g\in L^{\infty}$. Then $S_{f, g}$ is essentially normal if and only if for each support set $S$, one of the following holds:\\
$(i)$ $f|_{S}$ and $g|_{S}$ are constants;\\
$(ii)$ there is some unimodular constant $c$ such that$\left(g-cf\right)\big|_{S}$ is a constant, and
$$\left[(c-1)|f|^2+d\overline{f}-c\overline{d}f\right]\big|_{S}\in H^{\infty}|_{S},$$  
where $d=\left(g-cf\right)\big|_{S}$. 
\end{thm}
\begin{proof}
By the Lemma \ref{lem1}, we have
\begin{equation}\label{eq14}
\begin{aligned}
&\~~~\~~~\~~~S_{f, g}S^{*}_{f, g}-S^{*}_{f, g}S_{f, g}\\
&=\left (\begin{matrix}
  T_{f} & H_{\overline{g}}^* \\
    H_{f} & S_{g}
  \end{matrix}\right)\left (\begin{matrix}
    T_{\overline{f}} & H_{f}^* \\
    H_{\overline{g}} & S_{\overline{g}}
  \end{matrix}\right )-\left(\begin{matrix}
  T_{\overline{f}} & H_{f}^* \\
    H_{\overline{g}} & S_{\overline{g}}
  \end{matrix}\right)\left (\begin{matrix}
  T_{f} & H_{\overline{g}}^* \\
    H_{f} & S_{g}
  \end{matrix}\right)\\
  &=\left (\begin{matrix}
  T_{f}T_{\overline{f}}+H_{\overline{g}}^{*}H_{\overline{g}}-T_{\overline{f}}T_{f}-H^{*}_{f}H_{f}   &  T_{f}H_{f}^{*}+H_{\overline{g}}^{*}S_{\overline{g}}-T_{\overline{f}}H_{\overline{g}}^{*}-H_{f}^{*}S_{g} \\
  H_{f}T_{\overline{f}}+S_{g}H_{\overline{g}}-H_{\overline{g}}T_{f}-S_{\overline{g}}H_{f} & H_{f}H_{f}^{*}+S_{g}S_{\overline{g}}-H_{\overline{g}}H_{\overline{g}}^{*}-S_{\overline{g}}S_{g}
  \end{matrix}\right)\\ 
&=\left (\begin{matrix}
  H_{\overline{g}}^{*}H_{\overline{g}}-H^{*}_{\overline{f}}H_{\overline{f}} &  T_{f}H_{f}^{*}+H_{|g|^2}^{*}-T_{g}H^{*}_{g}-H_{f\overline{g}}^{*}\\
   H_{f}T_{\overline{f}}+H_{|g|^2}-H_{g}T_{\overline{g}}-H_{f\overline{g}} & H_{f}H_{f}^{*}-H_{g}H_{g}^{*}
  \end{matrix}\right).
 \end{aligned}
 \end{equation}
 The last equality follows from
 $$T_{f}T_{\overline{f}}-T_{\overline{f}}T_{f}=H_{f}^{*}H_f-H_{\overline{f}}^{*}H_{\overline{f}},$$
 $$H_{\overline{g}}^{*}S_{\overline{g}}=H_{|g|^2}^{*}-T_{g}H^{*}_{g},\ T_{\overline{f}}H_{\overline{g}}^{*}+H_{f}^{*}S_{g}=H_{f\overline{g}}^{*},$$
 $$S_{g}H{\overline{g}}=H_{|g|^{2}}-H_{g}T_{\overline{g}},\ S_{\overline{g}}H_{f}+H_{\overline{g}}T_{f}=H_{f\overline{g}},$$
 and
 $$S_{g}S_{\overline{g}}-S_{\overline{g}}S_{g}=H_{\overline{g}}H_{\overline{g}}^{*}-H_{g}H_{g}^{*}.$$
 Let\begin{equation}\label{eq13}\begin{aligned}
 K_1&= H_{\overline{g}}^{*}H_{\overline{g}}-H^{*}_{\overline{f}}H_{\overline{f}},\\
 K_2&=H_{f}T_{\overline{f}}+H_{|g|^2}-H_{g}T_{\overline{g}}-H_{f\overline{g}},\\
 K_3&=H_{f}H_{f}^{*}-H_{g}H_{g}^{*}.
 \end{aligned}
 \end{equation}
 Then 
 \begin{align*}S_{f, g}S^{*}_{f, g}-S^{*}_{f, g}S_{f, g}=\left (\begin{matrix}
    K_{1} & K_2 \\
    K_{2}^{*} & K_3
  \end{matrix}\right),
  \end{align*}
 which induces that $S_{f, g}$ is essentially normal if and only if $K_1$, $K_2$, and $K_3$ are compact.
 
 Firstly, we will show the sufficient part of this theorem.  We have by Lemma \ref{lem5} that
 $$K_{1}-T_{\phi_{z}}^{*}K_{1}T_{\phi_{z}}=V\left[(H_{\overline{g}}k_{z})\otimes(H_{\overline{g}}k_{z})-(H_{\overline{f}}k_{z})\otimes(H_{\overline{f}}k_{z})\right]V^{*},$$
 and we are going to prove that each condition in Theorem \ref{thm16} can implies that
 \begin{equation}\label{eq16}
 \lim\limits_{z\rightarrow m}\|K_{1}-T_{\phi_{z}}^{*}K_{1}T_{\phi_{z}}\|=0,
 \end{equation}
 for each $m\in \mathcal M(H^\infty+C)$. If condition $(i)$ holds, then
 $$\lim\limits_{z\rightarrow m}\|H_{\overline{f}}k_z\|_2=\lim\limits_{z\rightarrow m}\|H_{\overline{g}}k_z\|_2=0$$
 and by the triangle inequality, we have 
 $$\lim\limits_{z\rightarrow m}\left\|(H_{\overline{g}}k_{z})\otimes(H_{\overline{g}}k_{z})-(H_{\overline{f}}k_{z})\otimes(H_{\overline{f}}k_{z})\right\|=0,$$
 to obtain
 $$\lim\limits_{z\rightarrow m}\|K_{1}-T_{\phi_{z}}^{*}K_{1}T_{\phi_{z}}\|=0.$$
Now we suppose condition $(ii)$ holds, then
\begin{align*}
&\~~~~\~~~~(H_{\overline{g}}k_{z})\otimes(H_{\overline{g}}k_{z})-(H_{\overline{f}}k_{z})\otimes(H_{\overline{f}}k_{z})\\
&=[H_{\overline{\left(g-cf+cf\right)}}k_{z}]\otimes[H_{\overline{\left(g-cf+cf\right)}}k_{z}]-(H_{\overline{f}}k_{z})\otimes(H_{\overline{f}}k_{z})\\
&=[H_{\overline{\left(g-cf\right)}}k_{z}]\otimes[H_{\overline{\left(g-cf\right)}}k_{z}]+(H_{\overline{cf}}k_{z})\otimes[H_{\overline{\left(g-cf\right)}}k_{z}]\\
&\~~~~\~~~~+[H_{\overline{\left(g-cf\right)}}k_{z}]\otimes(H_{\overline{cf}}k_{z})+(H_{\overline{cf}}k_{z})\otimes(H_{\overline{cf}}k_{z})-(H_{\overline{f}}k_{z})\otimes(H_{\overline{f}}k_{z})\\
&=[H_{\overline{\left(g-cf\right)}}k_{z}]\otimes[H_{\overline{\left(g-cf\right)}}k_{z}]+\overline{c}(H_{\overline{f}}k_{z})\otimes[H_{\overline{\left(g-cf\right)}}k_{z}]+c[H_{\overline{\left(g-cf\right)}}k_{z}]\otimes(H_{\overline{f}}k_{z})
\end{align*}
and 
$$\left\|(H_{\overline{g}}k_{z})\otimes(H_{\overline{g}}k_{z})-(H_{\overline{f}}k_{z})\otimes(H_{\overline{f}}k_{z})\right\|\leq\left\|H_{\overline{\left(g-cf\right)}}k_{z}\right\|_{2}^{2}+2\|H_{\overline{f}}k_{z}\|_{2}\cdot\left\|H_{\overline{\left(g-cf\right)}}k_{z}\right\|_{2}.$$
Since $\left(g-cf\right)\big|_{S_m}$ is a constant for each support set $S_m$ of $m$, we obtain that $\lim\limits_{z\rightarrow m}\left\|H_{\overline{\left(g-cf\right)}}k_{z}\right\|=0$ and
 $$\lim\limits_{z\rightarrow m}\left\|(H_{\overline{g}}k_{z})\otimes(H_{\overline{g}}k_{z})-(H_{\overline{f}}k_{z})\otimes(H_{\overline{f}}k_{z})\right\|=0.$$
 On the other hand, noting
 $$K_1= H_{\overline{g}}^{*}H_{\overline{g}}-H^{*}_{\overline{f}}H_{\overline{f}}=T_{|g|^{2}}-T_{g}T_{\overline{g}}-T_{|f|^{2}}+T_{f}T_{\overline{f}},$$
 which is a finite sum of finite products of Toeplitz operators. Using the same method as in the proof of Proposition
 \ref{prop14}, we conclude by (\ref{eq16}) that $K_{1}$ is compact.
 
 Using
$$VT_{\varphi}=S_{\overline{\varphi}}V,\ VH_{\varphi}=H_{\varphi}^{*}V\ \mathrm{and}\ V^2=I$$
again,  we have
\begin{align*}\label{4.c}
VK_{3}V&=V\left(H_{f}H_{f}^{*}-H_{g}H_{g}\right)V\\
     &=H_{f}^{*}V^2H_{f}-H_{g}^{*}V^2H_{g}\\
     &=H_{f}^{*}H_{f}-H_{g}^{*}H_{g}.
\end{align*}
Using the same arguments as above, we conclude that
$$H_{f}^{*}H_{f}-H_{g}^{*}H_{g}$$
is compact, which gives us that $K_3$ is also compact. To complete the proof of the sufficient part, we need to 
show that $K_{2}$ is compact. Observe that the operator $K_2K_2^*$ is in the Toeplitz algebra $\mathscr{T}_{L^{\infty}}$ and the symbol map maps $K_{2}^{*}K_{2}$ to 0.
By \cite[Theorem 12]{GkZ2} again, we need only to prove that  $$\lim\limits_{z\rightarrow m} \|K_{2}^{*}K_{2}-T_{\phi_z}^{*}K_{2}^{*}K_{2}T_{\phi_z}\|=0.$$
Since
\begin{equation}\label{eq22}
\begin{aligned}
&\~~~~\~~~~\left(S_{\phi_{z}}H_{|g|^2}-H_{|g|^2}T_{\phi_{z}}\right)(h)\\
&=\left(I-P\right)\phi_{z}\left(I-P\right)\left(|g|^2h\right)-\left(I-P\right)\left(|g|^2\phi_{z}h\right)\\
&=\left(I-P\right)\left(\phi_{z}|g|^2h\right)-\left(I-P\right)\phi_{z}P\left(|g|^2h\right)-\left(I-P\right)\left(|g|^2\phi_{z}h\right)\\
&=0
\end{aligned}
\end{equation}
and
\begin{equation}\label{eq23}
\begin{aligned}
&\~~~~\~~~~\left(S_{\phi_{z}}H_{f\overline{g}}-H_{f\overline{g}}T_{\phi_{z}}\right)(h)\\
&=\left(I-P\right)\phi_{z}\left(I-P\right)\left(f\overline{g}h\right)-\left(I-P\right)\left(f\overline{g}\phi_{z}h\right)\\
&=\left(I-P\right)\left(\phi_{z}f\overline{g}h\right)-\left(I-P\right)\phi_{z}P\left(f\overline{g}h\right)-\left(I-P\right)\left(f\overline{g}\phi_{z}h\right)\\
&=0
\end{aligned}
\end{equation}
 for any $h\in H^{2}$, we have by Lemma \ref{lem7} that
 $$S_{\phi_{z}}K_2-K_2T_{\phi_{z}}=(H_{g}k_z)\otimes(VH_{\overline{g}}k_z)-(H_{f}k_z)\otimes(VH_{\overline{f}}k_z).$$
Let $$L_{z}=(H_{g}k_z)\otimes(VH_{\overline{g}}k_z)-(H_{f}k_z)\otimes(VH_{\overline{f}}k_z),$$
then condition $(i)$ or $(ii)$ can implies that 
\begin{equation}\label{eq17}
\lim\limits_{z\rightarrow m}\|L_z\|=0.
\end{equation}
Indeed, it is obvious that (\ref{eq17}) holds if $f|_{S_m}$ and $g|_{S_m}$ are constant for each support set $S_{m}$ of $m$. Else if condition $(ii)$ holds, then
\begin{equation}\label{eq18}
\begin{aligned}
&\~~~~\~~~~(H_{g}k_z)\otimes(VH_{\overline{g}}k_z)-(H_{f}k_z)\otimes(VH_{\overline{f}}k_z)\\
&=[H_{\left(g-cf+cf\right)}k_z]\otimes[VH_{\overline{\left(g-cf+cf\right)}}k_z]-(H_{f}k_z)\otimes(VH_{\overline{f}}k_z)\\
&=[H_{\left(g-cf\right)}k_z]\otimes[VH_{\overline{\left(g-cf\right)}}k_z]+c(H_{f}k_z)\otimes[VH_{\overline{\left(g-cf\right)}}k_z]\\
&\~~~~\~~~~+\overline{c}[H_{\left(g-cf\right)}k_z]\otimes(VH_{\overline{f}}k_z)+|c|^{2}(H_{f}k_z)\otimes(VH_{\overline{f}}k_z)-(H_{f}k_z)\otimes(VH_{\overline{f}}k_z)\\
&=[H_{\left(g-cf\right)}k_z]\otimes[VH_{\overline{\left(g-cf\right)}}k_z]-c(H_{f}k_z)\otimes[VH_{\overline{\left(g-cf\right)}}k_z]\\
&\~~~~\~~~~-\overline{c}[H_{\left(g-cf\right)}k_z]\otimes(VH_{\overline{f}}k_z)
\end{aligned}
\end{equation}
and
\begin{equation}\label{eq19}
\begin{aligned}
&\~~~~\~~~~\left\|(H_{g}k_z)\otimes(VH_{\overline{g}}k_z)-(H_{f}k_z)\otimes(VH_{\overline{f}}k_z)\right\|\\
&\leq\left\|H_{\left(g-cf\right)}k_z\right\|_{2}\cdot\left\|H_{\overline{\left(g-cf\right)}}k_z\right\|_{2}+|c|\cdot\left\|H_{f}k_z\right\|_{2}\cdot\left\|H_{\overline{\left(g-cf\right)}}k_z\right\|_{2}\\
&\~~~~\~~~~+|\overline{c}|\cdot\left\|H_{\left(g-cf\right)}k_z\right\|_{2}\cdot\left\|H_{\overline{f}}k_z\right\|_{2}\\
&=\left\|H_{\overline{\left(g-cf\right)}}k_z\right\|_{2}\cdot\left\|H_{\overline{\left(g-cf\right)}}k_z\right\|_{2}+\left\|H_{f}k_z\right\|_{2}\cdot\left\|H_{\overline{\left(g-cf\right)}}k_z\right\|_{2}\\
&\~~~~\~~~~+\left\|H_{\left(g-cf\right)}k_z\right\|_{2}\cdot\left\|H_{\overline{f}}k_z\right\|_{2}.
\end{aligned}
\end{equation}
Since $$\lim\limits_{z\rightarrow m}\left\|H_{\left(g-cf\right)}k_z\right\|_{2}=\lim\limits_{z\rightarrow m}\left\|H_{\overline{\left(g-cf\right)}}k_z\right\|_{2}=0,$$
we have by the inequality (\ref{eq19}) that $\lim\limits_{z\rightarrow m}\|L_{z}\|=0$. Furthermore,
\begin{equation}\label{eq24}
\begin{aligned}
T_{\phi_{z}}^{*}K_{2}^{*}K_{2}T_{\phi_{z}}&=\left(K_2T_{\phi_{z}}\right)^{*}K_{2}T_{\phi_{z}}\\
&=\left(S_{\phi_{z}}K_2-L_{z}\right)^{*}\cdot\left(S_{\phi_{z}}K_2-L_{z}\right)\\
&=K_{2}^{*}S_{\phi_{z}}^{*}S_{\phi_{z}}K_{2}-K_{2}^{*}S_{\phi_{z}}^{*}L_z-L_{z}^{*}S_{\phi_{z}}K_{2}
+L_{z}^{*}L_{z}\\
&=K_{2}^{*}\left(I-Vk_z\otimes Vk_z\right)K_{2}-K_{2}^{*}S_{\phi_{z}}^{*}L_z-L_{z}^{*}S_{\phi_{z}}K_{2}
+L_{z}^{*}L_{z}\\
&=K_{2}^{*}K_{2}-\left(K_{2}^{*}Vk_z\right)\otimes \left(K_{2}^{*}Vk_z\right)-K_{2}^{*}S_{\phi_{z}}^{*}L_z-L_{z}^{*}S_{\phi_{z}}K_{2}
+L_{z}^{*}L_{z}
\end{aligned}
\end{equation}
and
 \begin{align*}
&\~~~~\~~~~ \left\|K_{2}^{*}K_{2}-T_{\phi_z}^{*}K_{2}^{*}K_{2}T_{\phi_z}\right\|\\
&=\left\|\left(K_{2}^{*}Vk_z\right)\otimes \left(K_{2}^{*}Vk_z\right)+K_{2}^{*}S_{\phi_{z}}^{*}L_{z}+L_{z}^{*}S_{\phi_{z}}K_{2}-L_{z}^{*}L_{z}\right\|\\
 &\leq\left\|K_{2}^{*}Vk_z\right\|_{2}^{2}+\left\|K_{2}^{*}S_{\phi_{z}}^{*}L_{z}\right\|+\left\|L_{z}^{*}S_{\phi_{z}}K_{2}\right\|+\left\|L_{z}^{*}L_{z}\right\|,
\end{align*}
it is sufficient to show
\begin{align}\label{eq20}
\lim\limits_{z\rightarrow m} \|K_{2}^{*}Vk_z\|_2=0
\end{align}
as $\lim\limits_{z\rightarrow m}\|L_z\|=0$. For this purpose, we will check that condition $(ii)$ of 
Theorem \ref{thm16} can imply the equality (\ref{eq20}).
Computing $K_{2}^{*}Vk_{z}$ directly, we obtain
\begin{equation}\label{eq21}
\begin{aligned}
K_{2}^{*}Vk_{z}&=V\left(S_{\overline{f}}H_{f}+H_{|g|^{2}}k_{z}-S_{\overline{g}}H_{g}k_{z}-H_{f\overline{g}}k_{z}\right)k_z\\
&=V\left(S_{\overline{f}}H_{f}+H_{\overline{g}}T_{g}-H_{f\overline{g}}\right)k_z\\
&=V\left[S_{\overline{f}}H_{f}+H_{\overline{\left(g-cf+cf\right)}}T_{g}-H_{f\overline{g}}\right]k_z\\
&=V\left[S_{\overline{f}}H_{f}+H_{\overline{\left(g-cf\right)}}T_{g}+\overline{c}H_{\overline{f}}T_{\left(g-cf+cf\right)}-H_{f\overline{g}}\right]k_z\\
&=VH_{\overline{\left(g-cf\right)}}T_{g}k_{z}+V\left[S_{\overline{f}}H_{f}+H_{\overline{f}}T_{f}+\overline{c}H_{\overline{f}}T_{\left(g-cf\right)}-H_{f\overline{g}}\right]k_z\\
&=VH_{\overline{\left(g-cf\right)}}T_{g}k_{z}+V\left\{H_{|f|^{2}}+\overline{c}\left[H_{\overline{f}\left(g-cf\right)}-S_{\overline{f}}H_{\left(g-cf\right)}\right]-H_{f\overline{g}}\right\}k_z\\
&=VH_{\overline{\left(g-cf\right)}}T_{g}k_{z}-cVS_{\overline{f}}H_{\left(g-cf\right)}k_{z}\\
&\~~~~\~~~~+V\left[H_{|f|^{2}}+\overline{c}H_{\overline{f}\left(g-cf\right)}-H_{f\overline{\left(g-cf+cf\right)}}\right]k_z\\
&=VH_{\overline{\left(g-cf\right)}}T_{g}k_{z}-cVS_{\overline{f}}H_{\left(g-cf\right)}k_{z}\\
&\~~~~\~~~~+V\left[\left(1-\overline{c}\right)H_{|f|^{2}}+\overline{c}H_{\overline{f}\left(g-cf\right)}-H_{f\overline{\left(g-cf\right)}}\right]k_z\\
&=VH_{\overline{\left(g-cf\right)}}T_{g}k_{z}-cVS_{\overline{f}}H_{\left(g-cf\right)}k_{z}+cVH_{\left[\left(c-1\right)|f|^{2}+\overline{f}\left(g-cf\right)-cf\overline{\left(g-cf\right)}\right]}k_{z}.
\end{aligned}
\end{equation}  
Since $\left(g-cf\right)\big|_{S_{m}}=c$ and 
\begin{align*}
&\~~~~\~~~~\left[\left(c-1\right)|f|^{2}+\overline{f}\left(g-cf\right)-cf\overline{\left(g-cf\right)}\right]\big|_{S_{m}}\\
&=\left[\left(c-1\right)|f|^{2}+d\overline{f}-c\overline{d}f\right]\big|_{S_{m}}\in H^{\infty}\big|_{S_{m}},
\end{align*}
According to the triangle inequality, we have by Lemma \ref{lem9}, Lemma \ref{lem12} and (\ref{eq21}) that
$$\lim\limits_{z\rightarrow m}\left\|K_{2}^{*}Vk_{z}\right\|=0.$$ 
 
 Next we are going to show the necessary part of this theorem. By the Lemma \ref{lem5} and Lemma \ref{lem6}, we have
 $$\lim\limits_{z\rightarrow m} \|K_{1}-T_{\phi_z}^*K_{1}T_{\phi_z}\|=0$$
 and
$$\lim\limits_{z\rightarrow m}\left\|\big[(H_{\overline{g}}k_z)\otimes(H_{\overline{g}}k_z)-(H_{\overline{f}}k_z)\otimes(H_{\overline{f}}k_z)\big]\right\|=0,$$
for each $m\in \mathcal M(H^\infty+C)$. It follows that
\begin{equation}\label{eq15}
(H_{\overline{g}}k_z)\otimes(H_{\overline{g}}k_z)
=(H_{\overline{f}}k_z)\otimes(H_{\overline{f}}k_z)+\varepsilon(z),
\end{equation}
where the operator $\varepsilon(z)$ satisfies that $\lim\limits_{z\rightarrow m}\|\varepsilon(z)\|=0.$

In the following, we still use the same notation $\varepsilon(z)$ to denote the various terms such that $$\|\varepsilon(z)\|\rightarrow 0\ \ \ \ \ (z\rightarrow m)$$ for simplicity.

If $$\varliminf_{z\rightarrow m}\|H_{\overline{g}}k_z\|_2=0,$$ then we have by Lemma \ref{lem9} and (\ref{eq15}) that $\overline{g}\big|_{S_m}\in H^{\infty}|_{S_m}$ for each support set $S_m$ of $m$, and 
$$\lim\limits_{z\rightarrow m}\|H_{\overline{f}}k_{z}\|_2=0,$$
to imply that $\overline{f}\big|_{S_m}\in H^{\infty}|_{S_m}$.

Now we need to analyse the case of
$$\varliminf_{z\rightarrow m}\|H_{\overline{g}}k_z\|_2> 0.$$
By (\ref{eq15}), we have
\begin{eqnarray*}
\langle H_{\overline{g}}k_z,H_{\overline{g}}k_z\rangle H_{\overline{g}}k_z=\langle H_{\overline{f}}k_z,H_{\overline{g}}k_z\rangle H_{\overline{f}}k_z+\varepsilon(z).
\end{eqnarray*}
Thus there exists a constant $a(z)$ depending  on $z$ such that
$$H_{\overline{g}}k_z=a(z)H_{\overline{f}}k_z+\varepsilon(z),$$
where $a(z)$ satisfies that
$$|a(z)|=\left|\frac{\langle H_{\overline{f}}k_z,H_{\overline{g}}k_z\rangle }{\left\|H_{\overline{g}}k_z\right\|_2^2}\right|\leqslant \frac{\left\|f\right\|_{\infty}}{\varliminf\limits_{z\rightarrow m}\|H_{\overline{g}}k_z\|_2}$$
for all $z\in \mathcal{O}(m)\cap \mathbb D$, so $|a(z)|$ is  bounded for $z\in \mathcal{O}(m)\cap \mathbb D$. For $m\in \mathcal{M}(H^\infty+C)$,
 $\mathcal{O}(m)$  denotes a neighbourhood of it in $\mathcal M(H^\infty)$. By the boundedness of $a(z)$ and by the corona theorem, there exist a net $\{z_\beta\}$ and a constant $a\in \mathbb C$ such that 
$$\lim_{\beta}z_\beta=m\ \mathrm{and}\ \lim_{\beta}a(z_\beta)=a.$$
Therefore, by the the equivalence between conditions (2) and (3) in Lemma \ref{lem9}, we obtain 
$$H_{\overline{g}}k_z=aH_{\overline{f}}k_z+\varepsilon(z).$$
It follows that $\left(\overline{g}-a\overline{f}\right)\big|_{S_m}\in H^{\infty}|_{S_m}$. Hence we have
\begin{align*}
&\~~~~\~~~~(H_{\overline{g}}k_{z})\otimes(H_{\overline{g}}k_{z})\\
&=(H_{a\overline{f}}k_{z})\otimes(H_{a\overline{f}}k_{z})+\varepsilon(z)\\
&=|a|^2(H_{\overline{f}}k_{z})\otimes(H_{\overline{f}}k_{z})+\varepsilon(z)\\
&=(H_{\overline{f}}k_{z})\otimes(H_{\overline{f}}k_{z})+\varepsilon(z).
\end{align*}
Since 
$$\varliminf\limits_{z\rightarrow m}\|H_{\overline{f}}k_{z}\|>0,$$ 
we conclude that $|a|=1$. Similarly, we consider $K_3$ with the same argument and conclude that $\left(g-bf\right)\big|_{S_m}\in H^{\infty}|_{S_m}$ for some unimodular constant $b$, or  $f|_{S_m}$ and $g|_{S_m}$ are in $H^{\infty}|_{S_m}$. So we consider $K_2$ under the following four conditions:
\begin{enumerate}
    \item $\left(\overline{g}-a\overline{f}\right)\big|_{S_{m}}$, $\left(g-bf\right)\big|_{S_{m}}$ are in $H^{\infty}|_{S_{m}}$ for some unimodular constants $a$ and $b$;
    \item $\left(\overline{g}-a\overline{f}\right)\big|_{S_{m}}$, $f|_{S_{m}}$ and $g|_{S_{m}}$ are in $H^{\infty}|_{S_{m}}$ for some unimodular constant $a$;
    \item $\overline{f}|_{S_{m}}$, $\overline{g}|_{S_{m}}$ and $\left(g-bf\right)\big|_{S_{m}}$ are in $H^{\infty}|_{S_{m}}$ for some unimodular constant $b$;
    \item $\overline{f}|_{S_{m}}$, $\overline{g}|_{S_{m}}$, $f|_{S_{m}}$ and $g|_{S_{m}}$ are in $H^{\infty}$.
\end{enumerate}
The condition $(4)$ means that $f|_{S_{m}}$ and $g|_{S_{m}}$ are constants. Besides, conditions $(2)$ and $(3)$ imply that $\left(g-\overline{a}f\right)\big|_{S_{m}}$ and $\left(g-bf\right)\big|_{S_{m}}$ are constants, respectively. Thus we only need to consider $K_2$ under the condition $(1)$.

According to equalities (\ref{eq22}) and (\ref{eq23}),
we have by Lemma \ref{lem7} and Lemma \ref{lem8} that
$$S_{\phi_{z}}K_2-K_2T_{\phi_{z}}=(H_{g}k_z)\otimes(VH_{\overline{g}}k_z)-(H_{f}k_z)\otimes(VH_{\overline{f}}k_z),$$
and
$$\lim\limits_{z\rightarrow m}\left\|(H_{g}k_z)\otimes(VH_{\overline{g}}k_z)-(H_{f}k_z)\otimes(VH_{\overline{f}}k_z)\right\|=0.$$
Because $\left(\overline{g}-a\overline{f}\right)\big|_{S_{m}}$, $\left(g-bf\right)\big|_{S_{m}}$ are in $H^{\infty}|_{S_{m}}$ for some unimodular constants $a$ and $b$, we obtain that 
\begin{align*}
&\~~~~\~~~~\lim\limits_{z\rightarrow m}\left\|(H_{g}k_z)\otimes(VH_{\overline{g}}k_z)-(H_{f}k_z)\otimes(VH_{\overline{f}}k_z)\right\|\\
&=\lim\limits_{z\rightarrow m}\left\|(H_{bf}k_{z})\otimes(VH_{a\overline{f}}k_{z})-(H_{f}k_{z})\otimes(VH_{\overline{f}}k_{z})\right\|\\
&=\lim\limits_{z\rightarrow m}\left\|(ab-1)(H_{f}k_{z})\otimes(VH_{\overline{f}}k_{z})\right\|\\
&=\lim\limits_{z\rightarrow m}|ab-1|\cdot\|H_{f}k_{z}\|\cdot\|H_{\overline{f}}k_{z}\|\\
&=0.
\end{align*}
It follow that either $ab=1$, or $f|_{S_m}\in H^{\infty}|_{S_m}$, or $\overline{f}|_{S_m}\in H^{\infty}|_{S_m}$. 
If $f|_{S_m}$ or $\overline{f}|_{S_m}$ belongs to $H^{\infty}|_{S_m}$, then condition $(1)$ turns to be condition $(2)$ or $(3)$. Else if $ab=1$, then $b=\overline{a}$, which induces that
$\left(g-\overline{a}f\right)\big|_{S_m}\in H^{\infty}|_{S_m}$. It gives that $\left(g-\overline{a}f\right)\big|_{S_m}$ must be a constant. 

With the assumption that $\left(g-cf\right)\big|_{S_{m}}$ is a constant for some unimodular constant $c$, we need to prove that 
$$\left[(c-1)|f|^2+d\overline{f}-c\overline{d}f\right]\big|_{S_{m}}\in H^{\infty}|_{S_{m}}$$
where $d=\left(g-cf\right)\big|_{S_{m}}$. Since $K_{2}$ is compact, we obtain by Lemma \ref{lem6} that 
$$\lim\limits_{z\rightarrow m}\left\|K_{2}^{*}K_{2}-T_{\phi_{z}}^{*}K_{2}^{*}K_{2}T_{\phi_{z}}\right\|=0.$$
So we have by (\ref{eq24}) that 
$$\left\|K_{2}^{*}Vk_z\right\|_{2}^{2}\leq\left\|K_{2}^{*}S_{\phi_{z}}^{*}L_{z}\right\|+\left\|L_{z}^{*}S_{\phi_{z}}K_{2}\right\|+\left\|L_{z}^{*}L_{z}\right\|+\left\|K_{2}^{*}K_{2}-T_{\phi_{z}}^{*}K_{2}^{*}K_{2}T_{\phi_{z}}\right\|.$$
Recalling that $\lim\limits_{z\rightarrow m}\|L_{z}\|=0$ if $\left(g-cf\right)\big|_{S_{m}}$ is a constant, we conclude that
$$\lim\limits_{z\rightarrow m} \|K_{2}^{*}Vk_z\|_2=0.$$
Thus we have by Lemma \ref{lem9}, Lemma \ref{lem12} and (\ref{eq21}) that
$$\lim\limits_{z\rightarrow m}\left\|H_{\left[\left(c-1\right)|f|^{2}+\overline{f}\left(g-cf\right)-cf\overline{\left(g-cf\right)}\right]}k_{z}\right\|_{2}=0.$$
It follows that  
\begin{align*}
&\~~~~\~~~~\left[\left(c-1\right)|f|^{2}+d\overline{f}-c\overline{d}f\right]\big|_{S_{m}}\\
&=\left[\left(c-1\right)|f|^{2}+\overline{f}\left(g-cf\right)-cf\overline{\left(g-cf\right)}\right]\big|_{S_{m}}\in H^{\infty}\big|_{S_{m}},
\end{align*}
to complete this proof.
\end{proof}

\subsection*{Acknowledgments}
The author is grateful to the referee for valuable comments and suggestions.

\end{document}